\documentclass[11pt]{amsart}
\usepackage{epsfig}
\usepackage{graphicx}
\usepackage{amscd}
\usepackage{amsmath}
\usepackage{amsxtra}
\usepackage{amsfonts}
\usepackage{amssymb}

\newtheorem{theorem}{Theorem}[section]

\newtheorem{lemma}[theorem]{Lemma}
\newtheorem{proposition}[theorem]{Proposition}

\theoremstyle{definition}
\newtheorem{definition}[theorem]{Definition}
\newtheorem{remark}[theorem]{Remark}

\newtheorem*{questions}{Questions}
\newtheorem{example}[theorem]{Example}
\theoremstyle{remark}

\renewcommand{\theclaim}{\textup{\theclaim}}

\newtheorem*{acknowledgements}{Acknowledgements}

\numberwithin{equation}{section}

\def\openone

{\mathchoice

{\hbox{\upshape \small1\kern-3.3pt\normalsize1}}

{\hbox{\upshape \small1\kern-3.3pt\normalsize1}}

{\hbox{\upshape \tiny1\kern-2.3pt\SMALL1}}

{\hbox{\upshape \Tiny1\kern-2pt\tiny1}}}

\makeatletter

\newbox\ipbox

\newcommand{\ip}[2]{\left\langle #1\, , \,#2\right\rangle}
\newcommand{\diracb}[1]{\left\langle #1\mathrel{\mathchoice

{\setbox\ipbox=\hbox{$\displaystyle \left\langle\mathstrut
#1\right.$}

\vrule height\ht\ipbox width0.25pt depth\dp\ipbox}

{\setbox\ipbox=\hbox{$\textstyle \left\langle\mathstrut
#1\right.$}

\vrule height\ht\ipbox width0.25pt depth\dp\ipbox}

{\setbox\ipbox=\hbox{$\scriptstyle \left\langle\mathstrut
#1\right.$}

\vrule height\ht\ipbox width0.25pt depth\dp\ipbox}

{\setbox\ipbox=\hbox{$\scriptscriptstyle \left\langle\mathstrut
#1\right.$}

\vrule height\ht\ipbox width0.25pt depth\dp\ipbox}

}\right. }

\newcommand{\dirack}[1]{\left. \mathrel{\mathchoice

{\setbox\ipbox=\hbox{$\displaystyle \left.\mathstrut
#1\right\rangle$}

\vrule height\ht\ipbox width0.25pt depth\dp\ipbox}

{\setbox\ipbox=\hbox{$\textstyle \left.\mathstrut
#1\right\rangle$}

\vrule height\ht\ipbox width0.25pt depth\dp\ipbox}

{\setbox\ipbox=\hbox{$\scriptstyle \left.\mathstrut
#1\right\rangle$}

\vrule height\ht\ipbox width0.25pt depth\dp\ipbox}

{\setbox\ipbox=\hbox{$\scriptscriptstyle \left.\mathstrut
#1\right\rangle$}

\vrule height\ht\ipbox width0.25pt depth\dp\ipbox}

} #1\right\rangle}

\newcommand{\cj}[1]{\overline{#1}}

\newcommand{\bz}{\mathbb{Z}}
\newcommand{\M}{\mathcal{M}}

\newcommand{\br}{\mathbb{R}}
\newcommand{\bd}{\mathbb{D}}
\newcommand{\bc}{\mathbb{C}}
\newcommand{\bt}{\mathbb{T}}
\newcommand{\bn}{\mathbb{N}}

\def\blfootnote{\xdef\@thefnmark{}\@footnotetext}

\input xy
\xyoption{all}
\usepackage{amssymb}

\newcommand{\supp}[1]{\text{supp} (#1)}

\def\-{^{-1}}

\def\ty{\emptyset}

\begin{document}

\title[Affine fractals as boundaries and their harmonic analysis]{Affine fractals as boundaries and their harmonic analysis}
\author{Dorin Ervin Dutkay}
\blfootnote{Supported in part by the National Science Foundation.}
\address{[Dorin Ervin Dutkay] University of Central Florida\\
	Department of Mathematics\\
	4000 Central Florida Blvd.\\
	P.O. Box 161364\\
	Orlando, FL 32816-1364\\
U.S.A.\\} \email{ddutkay@mail.ucf.edu}

\author{Palle E.T. Jorgensen}
\address{[Palle E.T. Jorgensen]University of Iowa\\
Department of Mathematics\\
14 MacLean Hall\\
Iowa City, IA 52242-1419\\}\email{jorgen@math.uiowa.edu}

\thanks{} 

\begin{abstract}
     We introduce the notion of boundary representation for fractal Fourier expansions, starting with a familiar notion of spectral pairs for affine fractal measures. Specializing to one dimension, we establish boundary representations for these fractals. We prove that as sets these fractals arise as boundaries of functions in closed subspaces of the Hardy space $H^2$. By this we mean that there are lacunary subsets $\Gamma$ of the non-negative integers, and associated closed $\Gamma$-subspace in the Hardy space $H^2(\bd)$, $\bd$ denoting the disk, such that for every function $f$ in in $H^2(\Gamma)$, and for every point $z$ in $\bd$, $f(z)$ admits a boundary integral represented by an associated  measure $\mu$, with integration over $\supp{\mu}$ placed as a Cantor subset on the circle $\bt := \mbox{bd}(\bd)$.

    We study families of pairs:  measures $\mu$ and sets $\Gamma$ of lacunary form, admitting lacunary Fourier series in $L^2(\mu)$;  i.e., configurations $\Gamma$  arranged with a geometric progression of empty spacing, or missing parts, gaps. Given $\Gamma$, we find corresponding generalized Szeg\" o kernels $G_\Gamma$, and we compare them to the classical Szeg\" o kernel for $\bd$.

     Rather than the more traditional approach of starting with $\mu$, and then asking for  possibilities for sets $\Gamma$,  such that we get Fourier series representations, we turn the problem upside down; now starting instead with a countably infinite discrete subset $\Gamma$, and, within a new duality framework, we study the possibilities for choices of measures $\mu$. 
\end{abstract}
\maketitle \tableofcontents
\section{Introduction}\label{intr}
In earlier papers, a number of authors studied a family of fractals $X$,
and associated measures $\mu$ which arise as limits of iterated function
systems (IFS). This framework includes for example infinite convolutions,
and therefore Bernoulli measures.

The starting point is a finite family $F$ of affine contractive mappings, and
the measure $\mu$ then results as a consequence of a procedure of Hutchinson
\cite{Hut81}. The fractal $X$ will be the support of $\mu$. When the family $F$,
is suitably restricted, it was shown in \cite{JoPe98,MR1785282,MR1655832,MR1929508,DJ06} that the Hilbert
space $L^2(\mu)$ then possesses a Fourier basis of orthogonal exponentials
$\{e_\lambda : \lambda\in\Lambda\}$. The set $\Lambda$ of exponentials in such an orthogonal basis
will be called the {\it spectrum} of $\mu$.

  When a spectrum $\Lambda$ exists we say that $(\mu, \Lambda)$ is a {\it spectral
pair}, and there is a variety of results dealing with inverse spectral theory
in this setting. Indeed, these results have many applications as they open
up for the use of tools from Fourier analysis in the study of this family of
fractals. While the procedure was developed for fractal measures $\mu$ with
compact support in $\br^d$, for any $d$, there are a number of features that set
aside the case $d=1$, which will be the focus here. In this case, a
normalization may be chosen in such a way that the spectrum $\Lambda$ is
contained in the non-negative integers $\bn_0$. So when $(\mu, \Lambda)$ is a
spectral pair, and $\Lambda$ is chosen in this way, we get a natural isometric
embedding of $L^2(\mu)$ into the Hardy space $H^2(\bd)$ of analytic functions on
the complex disk $\bd$.

   In this paper we deal with the resulting boundary representations. This
study requires tools different from the classical theory. To see this note
that the support $X$ of $\mu$ may be placed on the boundary $\bt$ (one-torus) of the
disk. But in the fractal cases, $X$ has Lebesgue measure zero; recall the
normalized Lebesgue measure is Haar measure of $\bt$. By contrast, the classical boundary
limits for functions in $H^2$ (Markov-Primalov-Fatou) yield only boundary
limits almost everywhere (a.e) w.r.t. Lebesgue measure on $\bt$. Indeed, our
measures $\mu$ are typically singular with respect to Lebesgue measure, and
have Lebesgue measure zero. Nonetheless we prove that the fractals arise as
boundaries of closed subspaces $H^2(\Lambda)$ in the Hardy space $H^2$. To do
this we develop a family of reproducing kernels needed for the purpose. Our
kernels have infinite product representations.

A separate motivation for our paper comes from the study of systems of frame vectors in Hilbert space. Frames generalize more familiar notions of bases in Fourier analysis; see for example \cite{CaFi09,CaWe08}. Our focus here is on the case when both the Hilbert space and the choice of vectors are restricted. We take $L^2(\mu)$ for Hilbert space, and we take the vectors (functions) in $L^2(\mu)$ to be the familiar complex exponentials of Fourier analysis; hence Fourier frames. In some cases, we will arrive at orthogonal families, and in others not.

      It was recently discovered that an important problem in operator algebras, the Kadison-Singer conjecture \cite{KS59} is equivalent to intriguing open problems for frames, many with direct applications to signal processing; see e.g., \cite{CaWe08}; and further section 4 below for further details.

      Our present restricted context for frame computations appears to be a fertile ground for generating the kind of singular frames that are likely to have a bearing on Kadison-Singer in its frame incarnations. There are relatively more technical details involved in the search for examples of Fourier frames satisfying one or the other in the list of a priori frame estimates in the literature.  While our main results regarding boundary representations are of independent interest, we hope that they will also serve to throw light on important questions regarding Fourier frames.

We will use the following definitions:
\begin{definition}\label{defaf}
Let $R$ be a $d\times d$ expansive real matrix, i.e., all its eigenvalues have absolute value strictly bigger than one. Let $B$ be a finite subset of $\br^d$. We define the {\it affine iterated function system (IFS)} denoted $(R,B)$:
\begin{equation}
\tau_b(x)=R^{-1}(x+b),\quad(x\in\br^d,b\in B)
\label{eqaf1}
\end{equation}
The unique Borel probability measure $\mu_B$ with the property that 
\begin{equation}
\mu_B(E)=\frac{1}{\#B}\sum_{b\in B}\mu_B(\tau_b^{-1}(E)),
\label{eqaf2}
\end{equation}
for all Borel sets in $\br^d$ is called {\it the invariant measure} for the affine IFS $(R,B)$ (see \cite{Hut81}) for details. 
\end{definition}

\begin{definition}\label{defhad}
Let $R$ be a $d\times d$ matrix, and $B$, $L$ two finite subsets of $\br^d$. We call $(R,B,L)$ a {\it Hadamard system} if $\#B=\#L$ and the matrix 
\begin{equation}
\frac{1}{\sqrt{\#B}}\left(e^{2\pi i R^{-1}b\cdot l}\right)_{b\in B, l\in L}
\label{eqhad1}
\end{equation}
is unitary.
\end{definition}

\section{Kernels for subspaces of the Hardy space}\label{sec2}
 In this section we introduce the notion of boundary representation, and we prove that spectral pairs in one dimension admit  such representations. By this we mean that when a spectral pair $(\mu, \Gamma)$, in a general class, is given, then for every function $f$ in the $\Gamma$-subspace in the Hardy space $H^2$ of the disk $\bd$, and for every point $z$ in $\bd$, $f(z)$ admits a representation by a $\Gamma$-Szeg\" o kernel $G_\Gamma$, with integration over $\supp{\mu}$ placed as a Cantor subset on the circle $\bt := \mbox{bd}(\bd)$. Thus $\supp{\mu}$ placed on $\bt$ will be a boundary for the subspace $H^2(\Gamma)$, and integration is with respect to the fractal measure $\mu$ from the spectral pair.

   We then turn to families of spectral pairs given by sets $\Gamma$ of lacunary form, i.e., configurations arranged with a geometric progression of empty spacing, or a missing parts, gaps; lacunary Fourier series. For this case we show that our Szeg\" o kernel $G_\Gamma$ arises as a factor in the familiar and classical Szeg\" o kernel for $\bd$.
\begin{definition}\label{def1}
Following \cite{Arv98} and \cite{Rud87} we set $H^2=H^2(\bd)$ the space of analytic functions in $\bd$ 

$$f(z)=c_0+c_1z+c_2z^2+\dots,\quad(z\in\bd)$$
such that 
$$\sum_{n\in\bn_0}|c_n|^2=:\|f\|_{H^2}^2<\infty.$$
With the Szeg\" o kernel
$$k(z,\xi):=\frac{1}{1-\cj z\xi}\, (\in H^2), z,\xi\in\bd;$$
we then get 
\begin{equation}
f(z)=\ip{k(z,\cdot)}{f}_{H^2}
\label{eqd1.1}
\end{equation}
valid for all $f\in H^2$ and all $z\in\bd$. The relation \eqref{eqd1.1} is a simple instance of a reproducing kernel property. For the theory of reproducing kernels; see \cite{Aro50}, and also \cite{Arv98, AlLe08,ADV09} for a variety of applications.
\end{definition}

\begin{theorem}\label{th1}
Let $\mu$ be a probability measure on $\br$ and assume $\Gamma\subset \bn_0:=\{0,1,2,\dots\}$ is a spectrum for $\mu$. Then 
\begin{enumerate}
	\item The map $J:L^2(\mu)\rightarrow H^2$
	\begin{equation}
Je_\gamma=z^\gamma,\quad (\gamma\in\Gamma)
\label{eq1.1}
\end{equation}
extends to an isometric embedding of $L^2(\mu)$ into $H^2$. 
\item Define the map $G$ on $\bd\times\br$
\begin{equation}
G(z,x):=\sum_{\gamma\in\Gamma}\cj{z}^\gamma e_\gamma(x),\quad(z\in\bd,x\in\br)
\label{eq1.2}
\end{equation}
Then 
\begin{equation}
(Jf)(z)=\int f(x)\cj G(z,x)\,d\mu(x)=\ip{G(z,\cdot)}{f}_{L^2(\mu)},\quad(z\in\bd)
\label{eq1.3}
\end{equation}
\item Assume in addition that $\Gamma=R\Gamma+L$ for some $R\in\bn, R\geq 2$ and some finite set $L\subset \bn_0$ such that no two elements in $L$ are congruent modulo $R$. Then 
\begin{equation}
\cj G(z,x)=\prod_{n=0}^\infty\left(\sum_{l\in L}z^{R^nl}\cj e_l(R^nx)\right),\quad(z\in\bd,x\in\br).
\label{eq1.4}
\end{equation}
The infinite product is uniformly convergent for $z$ in a compact subsets of $\bd$ and $x\in\br$.
\end{enumerate}
\end{theorem}

\begin{proof}
Since $\{e_\gamma : \gamma\in\Gamma\}$ is an orthonormal basis in $L^2(\mu)$, (i) follows immediately.

(ii) Let 
\begin{equation}
k_z(\xi)=\frac{1}{1-\cj{z} \xi}\quad(z\in\bd,\xi\in\bt)
\label{eqp1.1}
\end{equation}
be the Szeg\" o kernel. We know that functions $F\in H^2$, can be recovered from their boundary values $F^\sharp$ by
\begin{equation}
F(z)=\int_{\bt}\cj{k_z(\xi)}F^\sharp(\xi)\,d\xi=\ip{k_z(\cdot)}{F}_{H^2},\quad(z\in\bd)
\label{eqp1.2} 
\end{equation}

Let $G(z,\cdot):=(J^*k_z)(\cdot)$ for $z\in\bd$. We have for $z\in\bd$:
$$\ip{G(z,\cdot)}{f}_{L^2(\mu)}=\ip{J^*k_z}{f}_{L^2(\mu)}=\ip{k_z}{Jf}_{H^2}=(Jf)(z).$$

It remains to prove \eqref{eq1.2}.

For $\gamma\in\Gamma$ and $z'\in\bd$:
$$\ip{e_\gamma}{G(z',\cdot)}_{L^2(\mu)}=\ip{e_\gamma}{J^*k_{z'}}_{L^2(\mu)}=\ip{Je_\gamma}{k_{z'}}_{H^2}=\ip{z^\gamma}{k_{z'}}_{H^2}=\cj{(z')^\gamma}.$$
Thus
$$G(z',\cdot)=\sum_{\gamma\in\Gamma}\cj{(z')^\gamma}e_\gamma(\cdot).$$

(iii) The condition implies that $0\in \Gamma$; otherwise take $a=\min\Gamma$ and since $a=Ra'+l$, we must have $ a'\leq a$, $a'\in\Gamma$, so $a=a'=0$. Since the elements of $L$ are incongruent modulo $R$, it follows that every $\gamma\in\Gamma$ can be written {\it uniquely} as $\gamma=R\gamma'+l$ for some $\gamma'\in\Gamma$ and $l\in L$. Then we have
$$\cj G(z,x)=\sum_{l\in L}\sum_{\gamma\in\Gamma}z^{R\gamma+l}\cj{e_{R\gamma+l}(x)}=\sum_{l\in L}z^l\cj{e_l(x)}\sum_{\gamma\in\Gamma}z^{R\gamma}\cj{e_{\gamma}(Rx)}=\left(\sum_{l\in L}z^l\cj{e_l(x)}\right)G(z^R,Rx).$$

Since 
$$|G(z,x)-1|\leq\sum_{\gamma\in\Gamma\setminus{0}}|z|^\gamma\leq\sum_{n\geq1}|z|^n=\frac{z}{1-|z|}$$
for all $z\in\bd$ and $x\in\br$ it follows that $G(z,x)-1$ converges to $0$ as $z\rightarrow0$ (since $0\in\Gamma$), uniformly in $x\in\br$.

Iterating the previous equality, and since $G(z^{R^n},x)$ converges to $1$ exponentially fast and uniformly for $z$ in a compact subset of $\bd$ and for $x\in\br$, (iii) follows.

\end{proof}

\begin{definition}\label{def2}
(i) Let $X$ be a compact subset of $[0,1]$. We shall also consider $X$ as a subset of $\bt=\br/\bz$ via the mapping $x\mapsto e_1(x)=e^{2\pi ix}$. We will further consider restrictions of functions $f$ defined on all of $\bc$ via the identification $f(e_1(x))=\tilde f(x)$ where $\tilde f$ is then a $\bz$-periodic function on the line $\br$, and we view both $\br$ and $\bt$ embedded in $\bc$ in the usual way. The notation $\tilde f$ will be implicit in the discussion below. 

(ii) Let $X$ be as above, and let $\mu$ be a Borel probability measure supported on $X$. Consider subsets $\Gamma$ of $\bn_0$. Set 
$$\mathfrak A_\Gamma:=\left\{{\sum_{\gamma\in\Gamma}}^{\mbox{(finite)}}c_\gamma z^\gamma : (c_\gamma)_{\gamma\in\Gamma}\mbox{ is a finite set of coefficients }\right\}.$$

(iii) We say that the pair $(\mu,\Gamma)$ has a {\it boundary representation} if there is a kernel function $k=k_{(\mu,\Gamma)}$ subject to the following conditions
\begin{enumerate}
\item[(a)] $k:\bd\times X\rightarrow\bc$.
\item[(b)] For all $z\in\bd$, $k(z,\cdot)\in L^2(X,\mu)$; and 
\item[(c)] For all $f\in \mathfrak A_\Gamma$, $z\in\bd$ we have 
\begin{equation}
f(z)=\int_X\cj{k(z,x)}\tilde f(x)\,d\mu(x)=\ip{k(z,\cdot)}{\tilde f}_{L^2(\mu)}.
\label{eq2.1}
\end{equation}
\end{enumerate}

We denote by $H^2(\Gamma)$ the subspace of $H^2$ spanned by the functions $z^\gamma$ with $\gamma\in\Gamma$.
\end{definition}

\begin{proposition}\label{pr3}
Let $(\mu,\Gamma)$ be a spectral pair, and assume that $\Gamma\subset\bn_0$; then this pair has a boundary representation with kernel $k=k_{(\mu,\Gamma)}$ given by 
\begin{equation}
\cj k(z,x)=\sum_{\gamma\in\Gamma} \cj{e_\gamma(x)}z^\gamma.
\label{eq2.2}
\end{equation}
Moreover then
\begin{equation}
\ip{k_z}{k_w}_{L^2(\mu)}=\sum_{\gamma\in\Gamma}(z\cj w)^{\gamma},\quad(z,w\in\bd).
\label{eq2.2.1}
\end{equation}

The representation \eqref{eq2.1} for functions in $\mathfrak A_\Gamma$ extends to $f\in H^2(\Gamma)$; moreover then $\tilde f\in L^2(X,\mu)$, and 
\begin{equation}
\|f\|_{H^2}=\|\tilde f\|_{L^2(\mu)}.
\label{eq2.6}
\end{equation}
\end{proposition}

\begin{proof}
If $f\in\mathfrak A_\Gamma$, we set $f(z)=\sum_{\gamma\in\Gamma}c_\gamma z^\gamma$, and note that the corresponding periodic function $\tilde f$ (as a restriction) satisfies 
\begin{equation}
\tilde f(x)=\sum_{\gamma\in\Gamma}c_\gamma e_\gamma(x).
\label{eq2.3}
\end{equation}
But then by restrictions $\tilde f\in L^2(X,\mu)$ and 
\begin{equation}
c_\gamma=\int_X \cj{e_\gamma(x)}\tilde f(x)\,d\mu(x),\quad(\gamma\in\Gamma)
\label{eq2.4}
\end{equation}
This is the $L^2(\mu)$-Fourier expansion implied by the assumption that $(\mu,\Gamma)$ is a spectral pair. Since the sum in \eqref{eq2.3} is finite, substitution of \eqref{eq2.4} yields
\begin{equation}
f(z)=\sum_\gamma c_\gamma z^\gamma=\sum_\gamma\int_X\cj{e_\gamma(x)}\tilde f(x)\,d\mu(x)\, z^\gamma=\int_X\sum_\gamma\cj{e_\gamma(x)}z^\gamma\tilde f(x)\,d\mu(x)=\int_X\cj{k(z,x)}\tilde f(x)\,d\mu(x),
\label{eq2.5}
\end{equation}
which is the kernel representation. 

The formula \eqref{eq2.2.1} follows if we make use of the ONB property of $\{e_\gamma : \gamma\in\Gamma\}$.

The argument further shows that formula \eqref{eq2.2} is the unique kernel function. The remaining properties follow from an application {\it mutatis mutandis} of the details in the proof of Theorem \ref{th1} above.

For functions $f\in H^2(\Gamma)$ by definition we have the unique representation
\begin{equation}
f(z)=\sum_{\gamma\in\Gamma}c_\gamma z^\gamma,\quad(z\in\bd)
\label{eq2.7}
\end{equation}
with
\begin{equation}
\|f\|_{H^2}^2=\sum_{\gamma\in\Gamma}|c_\gamma|^2<\infty.
\label{eq2.8}
\end{equation}
But since $(\mu,\Gamma)$ is a spectral pair, we have (by Parseval applied to $L^2(\mu)$):
\begin{equation}
\sum_{\gamma\in\Gamma}|c_\gamma|^2=\int_X|\tilde f|^2\,d\mu
\label{eq2.9}
\end{equation}
and
\begin{equation}
\tilde f(x)=\sum_{\gamma\in\Gamma}c_\gamma e_\gamma(x)
\label{eq2.10}
\end{equation}
holds as an $L^2(\mu)$-identity. This implies also \eqref{eq2.6}.

But \eqref{eq2.5} also holds $\mu$-a.e. when we pass to truncated summations on the right-hand side in \eqref{eq2.10}.

It remains to justify the exchange of summation and integration for the computation \eqref{eq2.5} when $f$ is now in $H^2(\Gamma)$($\subset H^2$). 

Now let $f\in H^2(\Gamma)$, and let $(c_\gamma)_{\gamma\in\Gamma}$ be the corresponding coefficients, see \eqref{eq2.4} and \eqref{eq2.7}. Using again Parseval in the form \eqref{eq2.6}, if $\epsilon>0$, there is a finite subset $F\subset\Gamma$ such that 
\begin{equation}
\sum_{\gamma\in\Gamma\setminus F}|c_\gamma|^2<\epsilon.
\label{eq2.11}
\end{equation}
Let $ f_F=\sum_{\gamma\in F} c_\gamma z^\gamma$. Then for $|z|<1$, using \eqref{eq2.1} for $f_F$:
$$\left|\int \cj k(z,x)\tilde f(x)\,d\mu(x)-\sum_{\gamma\in F}c_\gamma z^\gamma\right|\leq\left|\int \cj k(z,x)\tilde f(x)\,d\mu(x)-\int \cj k(z,x)\tilde f_F(x)\,d\mu(x)\right|$$$$+\left|\int \cj k(z,x)\tilde f_F(x)\,d\mu(x)-\sum_{\gamma\in F}c_\gamma z^\gamma\right|$$
$$\leq \|k_z\|_{L^2(\mu)}\|\tilde f-\tilde f_F\|_{L^2(\mu)}+0=\|k_z\|_{L^2(\mu)}\|f-f_F\|_{H^2}\rightarrow 0\mbox{ as } F\nearrow\Gamma.$$

\end{proof}

\begin{example}\label{rem4}
 There are differences between the boundary representation in the two cases, classical vs fractal, as we will see here with a simple example from \cite{JoPe98}. Referring to the Cantor construction with scale 4, we get a spectral pair $(\mu, \Gamma)$.  We consider the monomial $f(z) = z^2$ not  in the $\Gamma$-subspace subspace $H^2(\Gamma)$ in $H^2$ of the disk $\bd$. We sketch how $z^2$  is represented by the $\Gamma$-Szeg\" o kernel with integration over $\supp{\mu}$ placed as a Cantor subset on the circle $\bt$, and how it differs from the classical counterpart.

We caution that the representation of functions $f\in H^2(\Gamma)$ may differ from the more familiar $H^2$-boundary corresponding to the Haar (normalized Lebesgue) measure on $\bt$. The purpose of this example is to illustrate the significance of the isometric operator $J$ in \eqref{eq1.1} in Theorem \ref{th1}. Indeed the simple formula \eqref{eq1.1} is only valid for $\gamma\in\Gamma$. If $n\in\bn_0\setminus\Gamma$, then the function $e_n$ will typically be the boundary for a function different than $z^n$. 

To see this, take $(\mu,\Gamma)$ as follows: let $\mu$ be the invariant measure for the affine IFS with $R=4$ and $B=\{-1,1\}$, (see Definition \ref{defaf}  below) and let $\Gamma:=\{\sum_{k=0}^n 4^kl_k : l_k\in\{0,1\}, n\in\bn_0\}$. Then $(\mu,\Gamma)$ is a spectral pair \cite{JoPe98} and, for the Fourier transform, we have
\begin{equation}
\widehat\mu(t)=\prod_{n=1}^\infty\cos\left(\frac{2\pi t}{4^n}\right),\quad(t\in\br)
\label{eqr5.2}
\end{equation}
But then
$$J(e_2)=J\left(\sum_{\gamma\in\Gamma}\ip{e_\gamma}{e_2}e_\gamma\right)=\sum_{\gamma\in\Gamma}\widehat\mu(2-\gamma)z^\gamma\neq z^2.$$
A simple inspection of \eqref{eqr5.2} shows that $\widehat\mu(x)=\widehat\mu(-x)$ for $x\in\br$, $\widehat\mu(2)<0$, $\widehat\mu(2-16)\neq 0$ etc. Moreover, $\widehat\mu$ vanishes on odd integers.
\end{example}

\begin{definition}\label{def2.8}
 We say that $\Gamma$ is a {\it Riesz sequence} if there are constants $0<A_0\leq A_1<\infty$ such that: For all finite subsets of $\Gamma$ and all finitely indexed subsets $\{c_\gamma\}\subset\bc$, we have 
\begin{equation}
A_0\sum_{\gamma\in\Gamma}|c_\gamma|^2\leq \sum_{\gamma,\gamma'\in\Gamma}\cj c_\gamma c_{\gamma'}\widehat\mu(\gamma'-\gamma)\leq A_1\sum_{\gamma\in\Gamma}|c_\gamma|^2.
\label{eqr1}
\end{equation}
See also Proposition \ref{pr5.1} below and \cite{DHW10}.
\end{definition}

Returning to the operator $J$ from \eqref{eq1.1}, but in a more general framework than Theorem \ref{th1}, we have the following 
\begin{lemma}\label{lemr1}
Let $\mu$ be a a probability measure supported on $\bt$, and let $\Gamma$ be a Riesz sequence, and set $Je_\gamma=z^\gamma$ for all $\gamma\in\Gamma$, see \eqref{eq1.1}. Then $J$ extends to a bounded operator $L^2(\mu)\rightarrow H^2$ if and only if the lower estimate in \eqref{eqr1} holds for some $A_0>0$. In that case 
\begin{equation}
\|J\|_{L^2(\mu)\rightarrow H^2}\leq A_0^{-\frac12}.
\label{eqr2}
\end{equation}
Moreover, there is a bounded inverse operator $J^{-1}:H^2(\Gamma)\rightarrow L^2(\mu)$, $J^{-1}(z_\gamma)=e_\gamma$, $(\gamma\in\Gamma)$ if and only if the upper estimate in \eqref{eqr1} holds for some $A_1<\infty$. In that case 
\begin{equation}
\|J^{-1}\|_{H^2(\Gamma)\rightarrow L^2(\mu)}\leq \sqrt{A_1}.
\label{eqr3}
\end{equation}

\end{lemma}
\begin{proof}
This is standard operator theory. 
\end{proof}

\begin{proposition}\label{prr3}
Let $\mu$ and $\Gamma$ be as specified above, and assume that the Riesz sequence estimate \eqref{eqr1} holds. Let $J$ and $J^{-1}$ be the two bounded operators from Lemma \ref{lemr1}. Then the measure space $(\supp{\mu},\mathcal B,\mu)$ offers a boundary representation for $H^2(\Gamma)(\subset H^2)$, i.e., we get for all $f\in H^2(\Gamma)$, and all points $z\in\bd$
\begin{equation}
f(z)=\int_{\supp{\mu}}\cj{(J^*k_z)}(x)(J^{-1}f)(x)\,d\mu(x)
\label{eqr4}
\end{equation}
where $k_z$ in \eqref{eqr4} denotes the Szeg\" o kernel \eqref{eqp1.1}
\end{proposition}
\begin{proof}
By virtue of the assumption on the pair $(\mu,\Gamma)$, we get the two bounded operators $J$ and $J^{-1}$ as in Lemma \ref{lemr1}. Below we then compute adjoint operators with respect to the two Hilbert inner products in $H^2$, and in $L^2(\mu)$ respectively, denoted $\ip{\cdot}{\cdot}_{H^2}$ and $\ip{\cdot}{\cdot}_{L^2(\mu)}$ for emphasis. 

Let $f\in H^2(\Gamma)$ and $z\in\bd$ be fixed. Then 
\begin{align*}
f(z)&=\ip{k_z}{f}_{H^2} (\mbox{ by Szeg\" o see Definition \ref{def1}}) \\
 =&\ip{k_z}{JJ^{-1}f}_{H^2} ( \mbox{ by Lemma \ref{lemr1}})\\
 =&\ip{J^*k_z}{J^{-1}f}_{L^2(\mu)}=\int_{\supp{\mu}}\cj{(J^*k_z)}(x)(J^{-1}f)(x)\,d\mu(x),
\end{align*}
the desired conclusion \eqref{eqr4}.
\end{proof}
\begin{remark}\label{remr3}
The $\Gamma$-Szeg\" o kernel in \eqref{eqr4} is $G_\Gamma(z,\cdot)=J^*k_z$. Compare with the corresponding representation from Theorem \ref{th1}; this is the special case of Proposition \ref{prr3} for the case when $(\mu,\Gamma)$ is assumed to be a spectral pair, as opposed to merely a Riesz system.
\end{remark}

In the theorem below, we consider spectral pairs given by sets $\Gamma$ of lacunary form, i.e., configurations arranged with a choice of geometric progressions of empty spacing or gaps, similar to lacunary Fourier series. We then prove that our Szeg\" o kernels  $G_\Gamma$  arise as factors in the familiar and classical Szeg\" o kernel for $\bd$.

 We use the notation $A\oplus A'=\{0,\dots,R-1\}$ to indicate that every element $k\in\{0,\dots,R-1\}$ can be written uniquely as $k=a+a'$ with $a\in A$ and $a'\in A'$. We will also need the following Lemma:
 \begin{lemma}\label{lem6}\cite{MR0215807}
Suppose $A\oplus A'=\{0,\dots,R-1\}$ then one the following affirmations is true
\begin{enumerate}
	\item $A=\{0\}$ or $A'=\{0\}$.
	\item $1\in A$ and there exist a number $d\geq 2$ that divides $R$ and two subsets $C,C'$ of $\bn_0$ such that $A=dC\oplus\{0,\dots,d-1\}$, $A'=dC'$ and $C\oplus C'=\{0,\dots,R/d-1\}$.
	\item $1\in A'$ and (ii) holds with the roles of $A$ and $A'$ reversed.
\end{enumerate}

\end{lemma} 

\begin{theorem}\label{th5}
Suppose $A$ is a subset of $\bn_0$ such that there exists $A'\subset \bn_0$ and $R\in\bn, R\geq 2$ such that $A\oplus A'=\{0,\dots,R-1\}$ and $A,A'\neq\{0\}$. Then
\begin{enumerate}
	\item There exists finite subsets $L,L'\subset \{0,\dots,R-1\}$ such that $L\oplus L'=\{0,\dots,R-1\}$ and with the property that $(R,A,L)$ and $(R,A',L')$ are Hadamard systems.
 Also $\gcd(A)$ divides $R$. The set $L$ can picked such that $\gcd(A)\cdot\max(L)<R$. Similarly for $L'$. Here $\gcd(A)$ represents the greatest common divisor of $A$.
	\item Let $\mu_A$ be the invariant measure associated to the IFS $(R,A)$ and similarly for $\mu_{A'}$. Then the convolution $\mu_A\ast\mu_{A'}=\lambda|_{[0,1]}=$the Lebesgue measure restricted to $[0,1]$. 
	\item $\mu_A$ is spectral with spectrum $\Gamma(L)=\{\sum_{k=0}^n R^kl_k : l_k\in L, n\in\bn_0\}$ and similarly $\mu_{A'}$ is spectral with spectrum $\Gamma(L')$.
	\item The kernels satisfy the following relation 
	\begin{equation}
G_{\Gamma(L)}G_{\Gamma(L')}=k,
\label{eq5.1}
\end{equation}
\end{enumerate}
where $k$ is the classical Szeg\" o kernel.
\end{theorem}

\begin{proof}
(i) We proceed by induction on $R$. For $R=2$ this is trivial.

We use Lemma \ref{lem6}. The case (i) in this Lemma cannot occur under our hypotheses. Assume we are in the case (ii) of Lemma \ref{lem6}, case (iii) can be treated similarly. According to the induction hypothesis, there exist sets $M,M'$ such that $M\oplus M'=\{0,\dots,R/d-1\}$ and $(R/d,C,M)$ and $(R/d,C',M')$ are Hadamard systems. 

Define $L:=M\oplus\frac Rd\{0,\dots,d-1\}$ and $L'=M'$. Then it is easy to see that $L\oplus L'=\{0,\dots,R-1\}$ and $(R,A',L')$ is a Hadamard system (since $A'=dC'$). It remains to check that $(R,A,L)$ is a Hadamard system. 
We use the fact that $A=dC\oplus\{0,\dots,d-1\}$. Take $m,m'\in L$ and $j,j'\in \{0,\dots,d-1\}$. Then 
$$\sum_{c\in C}\sum_{i=0}^{d-1}e^{i 2\pi\frac{1}{R}(dc+i)((m-m')+\frac Rd(j-j'))}=\sum_{i=0}^{d-1}e^{i 2\pi\frac iR((m-m')+\frac Rd(j-j'))}\sum_{c\in C}e^{i 2\pi \frac{1}{R/d}c(m-m')},$$
because $\frac1R dc\frac Rd(j-j')$ is an integer.

If $m\neq m'$, then using the fact the $(R/d,C,M)$ is a Hadamard system, we get that the sum is zero. 
If $m=m'$ and $j\neq j'$ then using the fact that $(d,\{0,\dots,d-1\},\{0,\dots,d-1\})$ is a Hadamard system, we get again that the sum is zero.  

This proves that $(R,A,L)$ is a Hadamard system. 

It remains to prove the last statement. We proceed also by induction on $R$. If $R=4$ the result is easy to obtain. Since $A\neq\{0\}$ we have either $1\in A$ (and this case is trivial) or $A=dC$ for some $d\geq 2$, and $C\oplus C'=\{0,\dots,R/d-1\}$. 
Then as before, for the pair we can pick the dual sets $(M,M')$, and we can pick $L=M$. By the induction hypothesis $\gcd(C)\cdot\max(M)<R/d$ so 
$\gcd(A)\cdot\max(L)=d\gcd(C)\cdot\max(M)<R$. 
Also by the induction hypothesis $\gcd(C)$ divides $R/d$ so $\gcd(A)=d\gcd(C)$ divides $R$.

(ii) Let 
$$\chi_A(x)=\frac{1}{\#A}\sum_{a\in A}e^{2\pi iax},\quad \chi_{A'}(x)=\frac{1}{\#A'}\sum_{a'\in A'}e^{2\pi ia'x}$$
Since $A\oplus A'=\{0,\dots,R-1\}$ we can see that 
$$\chi_A(x)\chi_{A'}(x)=\frac{1}{R}\sum_{j=0}^{R-1}e^{2\pi ijx}=:\chi_{\{0,\dots,R-1\}}(x).$$

Then, using the infinite product formula for $\widehat\mu_A$ (see \cite{DJ06}), 
$$\widehat{\mu_A\ast\mu_{A'}}(x)=\widehat\mu_{A}(x)\widehat\mu_{A'}(x)=\prod_{n=1}^\infty\chi_A(R^{-n}x)\prod_{n=1}^\infty\chi_{A'}(R^{-n}x)=\prod_{n=1}^\infty\chi_A(R^{-n}x)\chi_{A'}(R^{-n}x)$$
$$=\prod_{n=1}^\infty\chi_{\{0,\dots,R-1\}}(R^{-n}x).$$
The change in the order of multiplication is allowed since the infinite products are uniformly convergent on compact sets.

The Lebesgue measure on $[0,1]$ is the invariant measure associated to the affine IFS $(R,\{0,\dots,R-1\})$. Therefore the last product above is the Fourier transform of $\lambda|_{[0,1]}$.

(iii) Using the results from \cite[Theorem 8.4]{DJ06}, we have to show there are no non-trivial extreme cycles (or $\chi_A$-cycles as they are called in \cite{DJ06}). Recall that a non-trivial extreme cycle is a finite set of non-zero points $\{x_0,x_1,\dots,x_{p-1},x_p:=x_0\}$ such that there exist $l_i\in L$ with $R^{-1}(x_i+l_i)=x_{i+1}$ for all $i\in\{0,\dots,p-1\}$, and such that 
$|\chi_A(x_i)|=1$ for all $i\in\{0,\dots,p-1\}$.

Assume by contradiction that there is such an extreme cycle. Since $|\chi_A(x_i)|=1$, and since $0\in A$, we must have equality in the triangle inequality so $e^{2\pi iax_i}=1$, which means that $x_i\in\frac1g\bz$, with $g=\gcd(A)$. Consider the smallest non-zero cycle, say  $x_0=k/g$. Then $\frac1R(\frac kg+l)$ is also a cycle point for some $l\in L$, so it is also of the form $k'/g$. From (i) we know $gl<R$. 

First, if $k\geq 2$ then 
$$\frac{k'}g=\frac1R(\frac kg+l)<\frac 1R(\frac kg+\frac Rg)\leq \frac{Rk}{Rg}=\frac kg=x_0,$$
and this would contradict the fact that $x_0$ is the smallest non-zero cycle. Then $k=1$, and using the computation above, with $k=1$, we get $\frac{k'}{g}<\frac 2g$, and therefore $k'=1$ too. But then we must have 
$$\frac1R(\frac 1g+l)=\frac 1g$$
so $1+gl=R$. Since $g$ divides $R$ (see (i)), we obtain that $g=1$. Then, since $A'\neq \{0\}$, we must have that $L'\neq \{0\}$ so $\max(L)<R-1$ and we get a contradiction.

In conclusion there are no non-trivial extreme cycles, hence with \cite{DJ06}, we get that $\Gamma(L)$ is a spectrum for $\mu_A$. Similarly for $\mu_{A'}$.

(iv) We have
$$G_{\Gamma(L)}(z,x)G_{\Gamma(L')}(z,x)=\sum_{\gamma\in\Gamma(L),\gamma'\in\Gamma(L')}\cj z^{\gamma+\gamma'}e_{\gamma+\gamma'}(x)$$
Note that $\Gamma(L)\oplus\Gamma(L')=\bn_0$. This can be seen from the base $R$ expansion of any natural number $n=\sum_{k=0}^pR^ka_k$, with $a_k\in\{0,\dots,R-1\}$. Since $L\oplus L'=\{0,\dots,R-1\}$, we get the unique decomposition as a sum $\gamma+\gamma'$ with $\gamma\in\Gamma(L)$ and $\gamma'\in\Gamma(L')$. 
Thus 
$$G_{\Gamma(L)}(z,x)G_{\Gamma(L')}(z,x)=\sum_{n=0}^\infty\cj z^ne_n(x)=\frac{1}{1-\cj ze_1(x)}=k(z,x).$$

\end{proof}

\section{Set-measure duality}\label{sec3}
 Most earlier studies of classes of spectral pairs $(\mu, \Gamma)$ have started with $\mu$, and then asked what possibilities there are for sets $\Gamma$ that make the two into a spectral pair; i.e., allow a Fourier series representation, typically with lacunary Fourier frequencies. In much of this work, the measures $\mu$ have been chosen at the outset to be self-similarity defined by a finite family of affine maps. In this section, we turn the problem upside down; starting with a countably discrete subset $\Gamma$, we ask what the possibilities are for choices of $\mu$. To do this we introduce a new duality framework.

\begin{definition}\label{def3.1}
Let the setting be as above, including dimension $d\geq1$; and consider
\begin{equation}
\M_1:=\left\{\mu: \mu\mbox{ is a Borel probability measure with compact support in }\br^d\right\}.
\label{eq3.1}
\end{equation}
We equip $\M_1$ with its weak$^*$ topology; and consider $\Gamma\subset\br^d$ some countable discrete subset. 
\begin{equation}
\M^\perp(\Gamma):=\left\{\mu\in\M_1 : \sum_{\gamma\in\Gamma}|\widehat\mu(t-\gamma)|^2\leq 1\right\}.
\label{eq3.2}
\end{equation}
If $A\geq1$, set 
\begin{equation}
\M_A(\Gamma):=\left\{\mu\in\M_1 : \sum_{\gamma\in\Gamma}|\widehat\mu(t-\gamma)|^2\leq A,\mbox{ for all } t\in\br^d\right\}.
\label{eq3.3}
\end{equation}
\begin{equation}
\M^{OB}(\Gamma):=\left\{\mu\in\M_1 : \sum_{\gamma\in\Gamma}|\widehat\mu(t-\gamma)|^2=1\mbox{ for all }t\in\br^d\right\}.
\label{eq3.4}
\end{equation}

Note that $\M^\perp(\Gamma)=\M_1(\Gamma)$.
\end{definition}

\begin{lemma}\label{lem3.2}

Fix $\Gamma$ as in the definition. Then
\begin{enumerate}
	\item $\mu\in\M^\perp(\Gamma)$ iff $\{e_\gamma\}_{\gamma\in\Gamma}$ is an orthogonal family in $L^2(\mu)$. 
	\item $\mu\in\M^{OB}(\Gamma)$ iff $\{e_\gamma\}_{\gamma\in\Gamma}$ is an ONB in $L^2(\mu)$. 
	\item If $\{e_\gamma\}_{\gamma\in\Gamma}$ forms a Bessel sequence in $L^2(\mu)$ with bound $A$, i.e., 
	$$\sum_{\gamma\in\Gamma}\left|\ip{e_\gamma}{f}_{L^2(\mu)}\right|^2\leq A\|f\|_{L^2(\mu)}^2,\quad(f\in L^2(\mu)),$$
then $\mu\in \M_A(\Gamma)$.
\end{enumerate}
\end{lemma}

\begin{proof}
For (i), (ii) see \cite{DJ06}. (iii) follows by applying the Bessel estimate to the functions $e_t$. 
\end{proof}

 Previously, the measures $\mu$ have been chosen at the outset to have self-similarity defined by a finite family of affine maps. In the theorem below, we turn the problem upside down, thus allowing for the possibility of any measure $\mu$. Hence our starting point is a fixed countably discrete subset $\Gamma$, and we ask what the possibilities are for choices of $\mu$. 

\begin{theorem}\label{th3.3}
Fix $\Gamma$ as in the definition, and $A\geq 1$. Then $\M_A(\Gamma)$ is a convex, weak$^*$-compact subset of $\M_1$. Same is true for $\M^\perp(\Gamma)$. The set $\M^{OB}(\Gamma)$ is contained in the extreme points of $\M^\perp(\Gamma)$. 
\end{theorem}

\begin{proof}
Let $\mu_1,\mu_2\in \M_A(\Gamma)$ and $\alpha\in[0,1]$ and set $\mu_\alpha:=\alpha\mu_1+(1-\alpha)\mu_2$; then using Schwarz's inequality on $l^2(\Gamma)$
$$\sum_{\gamma}|\widehat\mu_\alpha(t-\gamma)|^2\leq \alpha^2A+(1-\alpha)^2A+2\alpha(1-\alpha)\mbox{Re}\sum_{\gamma}\cj{\widehat{\mu}}_1(t-\alpha)\widehat\mu_2(t-\alpha)$$$$\leq A(\alpha^2+(1-\alpha)^2)+2\alpha(1-\alpha)\left(\sum_\gamma|\widehat\mu_1(t-\gamma)|^2\right)^{1/2}\left(\sum_\gamma|\widehat\mu_2(t-\gamma)|^2\right)^{1/2}$$$$\leq A(\alpha^2+(1-\alpha)^2+2\alpha(1-\alpha))=A.$$
Hence $\mu_\alpha$ is in $\M_A(\Gamma)$ and $\M_A(\Gamma)$ is convex. 

We check that $\M_A(\Gamma)$ is weak$^*$-closed. Take $\mu_n\in\M_A(\Gamma)$, and $\mu_n\rightarrow\mu\in\M_1$. Then, since $\widehat\mu(t)=\int e_t\,d\mu$, we have that $\lim_n\widehat\mu_n(t)=\widehat\mu(t)$ for all $t\in\br^d$. Using Fatou's lemma, we have
$$\sum_\gamma|\widehat\mu(t-\gamma)|^2\leq\liminf_n\sum_\gamma|\widehat\mu_n(t-\gamma)|^2\leq A,$$
so $\mu$ is in $\M_A(\Gamma)$, and therefore $\M_A(\Gamma)$ is weak$^*$-closed hence compact.

Since $\M^\perp(\Gamma)=\M_1(\Gamma)$, the same holds for $\M^\perp(\Gamma)$.

We check that points in $\M^{OB}(\Gamma)$ are extreme in $\M^\perp(\Gamma)$. For this, consider $\mu_\alpha,\mu_1,\mu_2$ and $\alpha$ as in the beginning of the proof, $A=1$, and assume $\mu_{\alpha}\in\M^{OB}(\Gamma)$; i.e., 
$$\sum_{\gamma}|\widehat\mu_\alpha(t-\gamma)|^2=1,\quad(t\in\br^d)$$
see Lemma \ref{lem3.2}. Using the same calculation it follows that we have equalities in all inequalities. In particular (assuming $0<\alpha<1$), we have that 

\begin{equation}
\sum_\gamma|\widehat\mu_1(t-\gamma)|^2=\sum_\gamma|\widehat\mu_2(t-\gamma)|^2=1,\quad(t\in\br^d).
\label{eq3.3.1}
\end{equation}

Also, we must have equality in the Schwarz inequality, so the vectors $(\widehat\mu_1(t-\gamma))_{\gamma\in\Gamma}$ and $(\widehat\mu_2(t-\gamma))_{\gamma\in\Gamma}$ in $l^2(\Gamma)$ are proportional, and since they both have norm one, the proportionality constant is $e^{i\theta}$ for some $\theta\in\br$. But the real part of the product must be equal to the absolute value, so $e^{i\theta}=1$. This implies that $\widehat\mu_1=\widehat\mu_2$ so $\mu_1=\mu_2=\mu_\alpha$. Hence $\mu$ is an extreme point. 
\end{proof}

 In earlier papers dealing with spectral pairs in one dimension, for example \cite{JoPe98} and \cite{DJ06}, one typically begins with a positive integer ($>$ 1) defining a scale similarity, for example an infinite convolution as in Example \ref{rem4} above. It is interesting to compare the two scale numbers 3 and 4 (the case in Example \ref{rem4}). If $\mu_3$ is the Cantor measure (i.e., for the ternary case), then it was shown in \cite{JoPe98} that $L^2(\mu_3)$ cannot have more than two orthogonal Fourier frequencies. By contrast, it was further shown in \cite{JoPe98}, that all the Cantor measures $\mu_m$, for $m$ even, are in the opposite extreme: they allow for spectra, i.e., admit sets $\Gamma_m$ such that $(\mu_m, \Gamma_m)$ is a spectral pair. In the example below, we turn around the question: we begin with a ternary choice for the set $\Gamma$ and then ask what possibilities there might be for $\mu$.

\begin{example}\label{ex3.4}
Set $d=1$ and 
$$\Gamma:=\left\{\sum_{i=0}^n a_i3^i : a_i\in\{0,1\}, n\in\bn_0\right\}.$$
Then $\M^\perp(\Gamma)=\M^\perp(\bz)$ and $\M^{OB}(\Gamma)=\ty$. 

To see this, take $\mu\in\M^\perp(\Gamma)$. Using the base-3 decomposition of positive integers using the digits $\{0,1,-1\}$, we see that $\Gamma-\Gamma=\bz$. So $\widehat\mu$ must vanish on $\bz\setminus\{0\}$. Therefore $\mu\in \M^\perp(\bz)$. Since $2\not\in\Gamma$ and $e_2\perp e_\gamma$ for all $\gamma\in\Gamma$, it follows that the set $\{e_\gamma : \gamma\in\Gamma\}$ cannot be complete. 
\end{example}

\section{Conclusions and open problems}\label{sec4}

While the literature on frame systems in Hilbert space is vast, see for example \cite{CaFi09,CaWe08}, and even in Banach space \cite{CaCh08}, our focus here is on the case when both the Hilbert space and the choice of vectors are restricted. We take the Hilbert space to be $L^2(\mu)$ where the family of measures is as outlined above, and we take the vectors to be the complex exponentials of Fourier analysis; Fourier frames.

      It was recently discovered that an important problem in operator algebras, the Kadison-Singer conjecture \cite{KS59} is equivalent to important open problems for frames; see e.g., \cite{CaWe08}. Our present restricted context for frame computations appears to be a fertile ground for generating the kind of singular frames that are likely to have a bearing on Kadison-Singer in its frame incarnations. But this means that there are relatively more technical details involved in the search for examples of Fourier frames satisfying one or the other of the frame estimates that subdivide the subject.  Below we include a table of cases, and an overview of what is known, and what is still open.

Let $\mu$ be in $\M_1$ and $\Gamma$ a discrete subset of $\br^d$. Define the function 
\begin{equation}
\sigma_{\Gamma}(t):=\sum_{\gamma\in\Gamma}|\widehat\mu(t-\gamma)|^2,\quad(t\in\br^d).
\label{eq5.5.1}
\end{equation}
and let 
\begin{equation}
E(\Gamma):=\{e_\gamma : \gamma\in\Gamma\}.
\label{eq5.5.2}
\end{equation}
The function $\sigma_\Gamma$ plays an central role in the study of sequences of exponential functions in $\br^d$. We review some of its properties here, and we list some open questions related to it.

\begin{proposition}\label{pr5.1}
Let $\mu$ and $\Gamma$ be as above. The function $\sigma_\Gamma$ has the following properties:
\begin{enumerate}
	\item $E(\Gamma)$ is an ONB for $L^2(\mu)$ if and only if $\sigma_\Gamma\equiv 1$. 
	\item $E(\Gamma)$ is an orthonormal set in $L^2(\mu)$ if and only if $\sigma_\Gamma\leq 1$.
	\item  $E(\Gamma)$ is a maximal orthonormal set of exponentials if and only if $0<\sigma_\Gamma\leq 1$.
	\item  If $E(\Gamma)$ is a Bessel sequence with bound $B>0$ then $\sigma_\Gamma\leq B$. 
	\item If $E(\Gamma)$ is a frame with bounds $A,B>0$ then $A\leq \sigma_\Gamma\leq B$. 
	\item $E(\Gamma)$ is a Riesz basic sequence with bounds $A,B>0$ if and only if the self-adjoint matrix 
	\begin{equation}
\mathcal G_\Gamma:=\left(\widehat\mu(\gamma-\gamma')\right)_{\gamma,\gamma'\in\Gamma}
\label{eq5.1.2}
\end{equation}
satisfies $AI_{l^2(\Gamma)}\leq \mathcal G_\Gamma\leq BI_{l^2(\Gamma)}$. 
\end{enumerate}
\end{proposition}

The statements (i), (ii), and (iv) just repeat Lemma \ref{lem3.2}. For a proof of (v), (vi) see \cite{DJ06,DHSW10,DHW10}. To prove (iii), we see that if $\sigma_\Gamma>0$, and if $\gamma'\not\in\Gamma$, then there is a $\gamma\in\Gamma$ such that $\widehat\mu(\gamma-\gamma')\neq 0$, so $e_{\gamma'}$ is not orthogonal to $e_\gamma$. Hence $E(\Gamma)$ is maximal. Conversely, if this set is maximal, then we cannot have $\sigma_\Gamma(t)=0$ because that would imply that $e_t$ is orthogonal to all $e_\gamma$ with $\gamma\in\Gamma$.

Here are some known results related to Proposition \ref{pr5.1}. We denote by $\mu_4$ the measure in the Jorgensen-Pedersen example \cite{JoPe98}, i.e., 
the invariant measure for the affine IFS with $R=4$ and $B=\{0,2\}$, and by $\mu_3$ the middle third Cantor measure, i.e., $R=3$, $B=\{0,2\}$. 

\begin{enumerate}
\item \cite{JoPe98,DJ06,DHS09} There are infinitely many sets $\Gamma$ that contain $0$ such that $E(\Gamma)$ is an ONB for $L^2(\mu_4)$. 
\item \cite{JoPe98} There are no sets $\Gamma$ with 3 or more elements such that $E(\Gamma)$ is orthogonal in $L^2(\mu_3)$. 
\item \cite{DHS09} There are maximal sets of orthogonal exponentials in $L^2(\mu_4)$ which are not ONBs for $L^2(\mu_4)$. 
\item \cite{DHW10} There are sets $\Gamma$ of positive Beurling dimension such that $E(\Gamma)$ is a Bessel sequence in $L^2(\mu_3)$.
\item \cite{DHW10} There are sets $\Gamma$ of positive Beurling dimension such that $E(\Gamma)$ is a Riesz basic sequence in $L^2(\mu_3)$.
\end{enumerate}

Here is a list of questions belonging to the same circle of ideas:
\begin{questions} The following question are still open at the time this paper was written:
\begin{enumerate}
	\item Does the converse of Proposition \ref{pr5.1} (iv) hold?
	\item Does the converse of Proposition \ref{pr5.1} (v) hold?
	\item Are there any sets $\Gamma$ such that $E(\Gamma)$ is a frame for $L^2(\mu_3)$?
	\item Are there any sets $\Gamma$ such that $E(\Gamma)$ is a Riesz basis for $L^2(\mu_3)$?
	
\end{enumerate}
\end{questions}

\begin{acknowledgements}
We are grateful to an anonymous referee for a list of very helpful suggestions. 
\end{acknowledgements}
 \bibliographystyle{alpha}
\bibliography{kernel}

\def\cprime{$'$}
\begin{thebibliography}{DHSW10b}

\bibitem[ADV09]{ADV09}
Daniel Alpay, Aad Dijksma, and Dan Volok.
\newblock Schur multipliers and de {B}ranges-{R}ovnyak spaces: the multiscale
  case.
\newblock {\em J. Operator Theory}, 61(1):87--118, 2009.

\bibitem[AL08]{AlLe08}
Daniel Alpay and David Levanony.
\newblock On the reproducing kernel {H}ilbert spaces associated with the
  fractional and bi-fractional {B}rownian motions.
\newblock {\em Potential Anal.}, 28(2):163--184, 2008.

\bibitem[Aro50]{Aro50}
N.~Aronszajn.
\newblock Theory of reproducing kernels.
\newblock {\em Trans. Amer. Math. Soc.}, 68:337--404, 1950.

\bibitem[Arv98]{Arv98}
William Arveson.
\newblock Subalgebras of {$C^*$}-algebras. {III}. {M}ultivariable operator
  theory.
\newblock {\em Acta Math.}, 181(2):159--228, 1998.

\bibitem[CC08]{CaCh08}
Peter~G. Casazza and Ole Christensen.
\newblock The reconstruction property in {B}anach spaces and a perturbation
  theorem.
\newblock {\em Canad. Math. Bull.}, 51(3):348--358, 2008.

\bibitem[CF09]{CaFi09}
Peter~G. Casazza and Matthew Fickus.
\newblock Minimizing fusion frame potential.
\newblock {\em Acta Appl. Math.}, 107(1-3):7--24, 2009.

\bibitem[CW08]{CaWe08}
Peter~G. Casazza and Eric Weber.
\newblock The {K}adison-{S}inger problem and the uncertainty principle.
\newblock {\em Proc. Amer. Math. Soc.}, 136(12):4235--4243, 2008.

\bibitem[DHS09]{DHS09}
Dorin~Ervin Dutkay, Deguang Han, and Qiyu Sun.
\newblock On the spectra of a {C}antor measure.
\newblock {\em Adv. Math.}, 221(1):251--276, 2009.

\bibitem[DHSW10a]{DHW10}
Dorin~Ervin Dutkay, Deguang Han, Qiyu Sun, and Eric Weber.
\newblock Bessel sequences of exponentials on fractal measures.
\newblock {\em preprint}, 2010.

\bibitem[DHSW10b]{DHSW10}
Dorin~Ervin Dutkay, Deguang Han, Qiyu Sun, and Eric Weber.
\newblock On the {B}eurling dimension of exponential frames.
\newblock {\em to appear in Adv. Math.}, 2010.

\bibitem[DJ06]{DJ06}
Dorin~Ervin Dutkay and Palle E.~T. Jorgensen.
\newblock Iterated function systems, {R}uelle operators, and invariant
  projective measures.
\newblock {\em Math. Comp.}, 75(256):1931--1970 (electronic), 2006.

\bibitem[Hut81]{Hut81}
John~E. Hutchinson.
\newblock Fractals and self-similarity.
\newblock {\em Indiana Univ. Math. J.}, 30(5):713--747, 1981.

\bibitem[JP98]{JoPe98}
Palle E.~T. Jorgensen and Steen Pedersen.
\newblock Dense analytic subspaces in fractal {$L\sp 2$}-spaces.
\newblock {\em J. Anal. Math.}, 75:185--228, 1998.

\bibitem[KS59]{KS59}
Richard~V. Kadison and I.~M. Singer.
\newblock Extensions of pure states.
\newblock {\em Amer. J. Math.}, 81:383--400, 1959.

\bibitem[Lon67]{MR0215807}
Calvin~T. Long.
\newblock Addition theorems for sets of integers.
\newblock {\em Pacific J. Math.}, 23:107--112, 1967.

\bibitem[{\L}W02]{MR1929508}
Izabella {\L}aba and Yang Wang.
\newblock On spectral {C}antor measures.
\newblock {\em J. Funct. Anal.}, 193(2):409--420, 2002.

\bibitem[Rud87]{Rud87}
Walter Rudin.
\newblock {\em Real and complex analysis}.
\newblock McGraw-Hill Book Co., New York, third edition, 1987.

\bibitem[Str98]{MR1655832}
Robert~S. Strichartz.
\newblock Remarks on: ``{D}ense analytic subspaces in fractal {$L\sp
  2$}-spaces'' [{J}.\ {A}nal.\ {M}ath.\ {\bf 75} (1998), 185--228; {MR}1655831
  (2000a:46045)] by {P}. {E}. {T}. {J}orgensen and {S}. {P}edersen.
\newblock {\em J. Anal. Math.}, 75:229--231, 1998.

\bibitem[Str00]{MR1785282}
Robert~S. Strichartz.
\newblock Mock {F}ourier series and transforms associated with certain {C}antor
  measures.
\newblock {\em J. Anal. Math.}, 81:209--238, 2000.

\end{thebibliography}

\end{document}